\documentclass[11pt]{amsart}
\usepackage{amssymb}
\usepackage{amsmath}
\usepackage{amsthm}
\usepackage{mathrsfs}
\usepackage{faktor}
\usepackage{tikz}
\usepackage{thmtools}
\usepackage{mathtools}

\usepackage{hyperref}

\usetikzlibrary{arrows}
\usetikzlibrary{shapes.geometric}
\usetikzlibrary{decorations.pathreplacing,decorations.markings}
\usepackage[utf8]{inputenc}
\usepackage{enumitem}
\usepackage{wasysym}

\usepackage[capitalise]{cleveref}

\newtheorem{theorem}{Theorem}[section]
\newtheorem{lem}[theorem]{Lemma}

\newtheorem{cor}[theorem]{Corollary}
\newtheorem{qu}[theorem]{Question}

\theoremstyle{definition}
\newtheorem{dfn}[theorem]{Definition}

\numberwithin{theorem}{section}

\DeclareMathOperator{\Mon}{Mon}

\DeclareMathOperator{\BS}{BS}

\newcommand{\Z}{\mathbb{Z}}

\title{The Diophantine problem in the Baumslag-Gersten group}
\subjclass[2020]{20F05, 20F65, 20F10}
\keywords{Diophantine problem, one-relator group, Baumslag-Gersten group}

\thanks{The research of the authors was supported by the EPSRC Fellowship grant EP/V032003/1 ‘Algorithmic, topological and geometric aspects of infinite groups, monoids and inverse semigroups’. }

\begin{document}

\maketitle

\vspace{-4mm}

\begin{center}
 ROBERT D. GRAY
 \footnote{School of Engineering, Mathematics, and Physics, 
University of East Anglia, Norwich NR4 7TJ, England.
 Email \texttt{Robert.D.Gray@uea.ac.uk}.
 }
 and
 ALEX LEVINE
 \footnote{School of Engineering, Mathematics, and Physics, University of 
East Anglia, Norwich NR4 7TJ, England.
 Email \texttt{A.Levine@uea.ac.uk}.
 }
\end{center}

\vspace{-4mm}

\begin{abstract}
We show that, for every positive prime number $p$, the Baumslag-Gersten group $G_p = \langle a, t \mid (tat^{-1}) a(tat^{-1})^{-1} = a^p \rangle$ has an undecidable Diophantine problem. This gives the first known examples of one-relator groups with undecidable Diophantine problems. Furthermore we apply this to show that the one-relator group $G_p \ast \mathbb{Z}$ has an undecidable single equation problem.  
\end{abstract}

\section{Introduction}
The study of one-relator groups has long been a central topic in
	combinatorial and geometric group theory and has seen a resurgence in
	recent years, with various historic conjectures having been resolved. 
In spite of this, algorithmic results remain relatively few; Magnus
	famously proved they have decidable word problem in the 1930s
	\cite{Magnus32}, and the first author recently showed that there exist
	one-relator groups with undecidable submonoid membership problems
	\cite{Gray20}, but the conjugacy, subgroup membership and isomorphism
	problems all remain open in general (see, for example
	\cite{CFMarcoOneRelatorSurvey}). There have been a good number of positive
	results for special cases, with the class of one-relator groups with
	torsion being particularly well-behaved.  

	In this paper we consider two decision problems that both generalise the conjugacy
	problem: the \emph{Diophantine problem} and the \emph{single equation
	problem}. The former asks if a given finite system of equations in a group
	admits a solution and the latter asks if a single equation admits a
	solution. The first major breakthrough in the study of Diophantine problems
	in groups was a positive answer for the classes of free semigroups and free
	groups by Makanin \cite{Makanin_systems, Makanin_semigroups,
	Makanin_eqns_free_group}. This was later generalised to torsion-free
	hyperbolic groups by Rips and Sela \cite{RipsSela}, and later to all
	hyperbolic groups by Dahmani and Guirardel \cite{dahmani_guirardel}.  
In addition to these results, there are various
	other classes of groups known to have decidable \cite{Dahmani09,
	DiekertMuscholl, duchin_liang_shapiro, Ersov72, GarettaMetabelian, GarettaRandomNilpotent, Levine2022} and undecidable \cite{Dong25,
			duchin_liang_shapiro, ElliottLevine, 
			Garreta2020, GrayLevine25, almost-1rel-DP,
			KharlampovichMiasnikov2025, romankov_undecidable_first,
	RomankovMetabelian} Diophantine and single equation problems. 

The Diophantine problem for one-relator groups has received attention in the
literature recently.  Some of the general positive results described in the
previous paragraph may be applied to show that it is decidable in certain
classes of one-relator groups e.g.  those that are hyperbolic, which includes
all one-relator groups with torsion.  As some one-relator groups are virtually
direct products of hyperbolic groups these also have decidable Diophantine
problems \cite{CiobanuGrayLevineInPrep, CiobanuHoltRees,
CFOneRelatorDiophantine}.  While the Diophantine problem for Baumslag--Solitar
groups remains open in general, in \cite{kharlampovich2020diophantine} it was
shown that the Diophantine problem for quadratic equations in $\BS(1, n)$ is
decidable. 
See also  
\cite{Duncan2023}, \cite{Mandel2023} and \cite{Mandel2023b} for other related work on 
Diophantine problem in $\BS(1,n)$.
The related question of the Diophantine problem for one-relator monoids has also been investigated e.g. in  
\cite{Kharlampovich2019, Garreta2021}. In particular, in 
\cite{Garreta2021} 
it is shown that the open problem of decidability of word equations with length constraints reduces to the Diophantine problem in certain one-relator monoids.
In \cite{CFOneRelatorDiophantine}, Nyberg Brodda explicitly poses the question of whether the Diophantine
	problem is decidable in all one-relator groups. 
The main result of this paper is to  
prove that the Baumslag-Gersten groups
provide a negative answer to this question.
\begin{theorem}
	\label{main-thm}
	Let \(p \in \Z_{> 0}\) be prime and \(G_p = \langle a, t \mid
	(tat^{-1}) a(tat^{-1})^{-1} = a^p \rangle\). Then the Diophantine problem
	in \(G_p\) is undecidable. 
\end{theorem}

	The Baumslag-Gersten group \(G_2\) was first introduced by Baumslag as an example
	of a one-relator group that is not residually finite, by showing that all
	finite quotient groups are cyclic \cite{Baumslag69}. Gersten later proved
	that the Dehn function grows particularly fast \cite{Gersten92} and
	Gersten's bound was improved upon by Platonov \cite{Platonov04}.
	Despite this `strange' behaviour, the Baumslag-Gersten group is more
	tame in some regards. Beese
	showed that conjugacy problem is decidable \cite{BeeseThesis} and
	Myasnikov, Ushakov and Won showed that the word problem is decidable
	in polynomial time \cite{MyasnikovUshakovWon}. Their method involved
	reducing words to sequences of letters and power circuits. Related
	to power circuits is the signature \((\Z_{> 0}, +, x \cdot 2^y, \leq)\),
	which was recently showed to have an undecidable Diophantine problem
	\cite{Rybalov}. 
While those results do not immediately imply that the Baumslag-Gersten group has an undecidable Diophantine problem, they might be regarded as evidence that the Baumslag-Gersten group should be a good candidate for this question.  
Our proof does not use any of the above results on power circuits but 
instead our approach is to prove a reduction to a result of Denef
	\cite{Denef}.

	By reducing the Diophantine problem in a group \(G\) to a single equation
	in \(G \ast \Z\), we are also able to obtain a one-relator group where
	the single equation problem is undecidable.

\begin{cor}
	\label{main-cor}
	Let \(p \in \Z_{> 0}\) be prime. Then the single equation problem in the
	one-relator group
	\(\langle a, c, t  \mid
	(tat^{-1}) a(tat^{-1})^{-1} = a^p \rangle\) is undecidable.
\end{cor}

\section{Preliminaries}
\label{sec:prelim}

\subsection{Diophantine and single equation problems}
	
	We begin by formally defining equations in groups.

    \begin{dfn}
		Let \(G\) be a finitely generated group and \(\mathcal X\) be a finite
		set. We call the elements of \(\mathcal X\) \emph{variables}. An
		\emph{equation} in \(G\) is an element \(w \in G \ast F(\mathcal X)\),
		which we denote by \(w = 1\). A \emph{solution} to \(w = 1\) is a
		homomorphism \(\phi \colon G \ast F(\mathcal X) \to G\), such that
		\(\phi(w)  = 1\) and \(\phi (g)= g\) for all \(g \in G\).  A
		\emph{(finite) system of equations} in \(G\) is a finite collection of
		equations in \(G\) using the same set of variables. A \emph{solution}
		to a system \(\mathcal E\) of equations is a homomorphism that is a
		solution to each equation in \(\mathcal E\). If \(\mathcal E\) is a
		system of equations, we will use the operator \(\wedge\) to delimit the
		equations in the system \(\mathcal E\); that is \(\mathcal E = (w_1 =
		1) \wedge (w_2 = 1) \wedge \cdots \wedge (w_n = 1)\).

		The \emph{Diophantine problem} (resp. \emph{single equation problem})
		in \(G\) is the decision question which asks if there is an algorithm
		which takes a system \(\mathcal E\) of equations (resp. single equation
		\(w = 1\)) in \(G\) as input and outputs whether or not \(\mathcal E\)
		(resp. \(w = 1\)) admits a solution. 
	\end{dfn}

Frequently, we will consider a solution to an equation
		with \(n \in \mathbb{Z}_{>0}\) variables \(X_1, \ldots, X_n\) in a
		group \(G\) to be an \(n\)-tuple of elements \((g_1,  \ldots,  g_n)\).
		The homomorphism \(\phi\) can be recovered from a
		tuple by defining \(\phi\) to be the unique homomorphism satisfying
		\(\phi(g)  = g\) for all \(g \in G\) and \(\phi(X_i) = g_i\) for all
		\(i\). We additionally abuse notation by considering equations of the
		form \(w_1 = w_2\), rather than
		\(w_1w_2^{-1} = 1\).

\subsection{Positive existential theories (with constants)}

	The proof of \cref{main-thm} requires us to look to the positive existential
	theory (with constants) in \(G\). This is a similar algorithmic question to
	the Diophantine problem and more importantly is algorithmically equivalent.

	\begin{dfn}
		Let \(G\) be a finitely generated group. A \emph{positive existential
		sentence (with constants)} \(\mathcal F\) in \(G\) is a finite
		disjunction of systems of equations in \(G\). A \emph{solution} to
		\(\mathcal F\) is any homomorphism which is a solution to at least one
		system of equations in \(\mathcal F\). The \emph{positive existential
		theory (with constants)} in \(G\) is the decision question asking if a
		given positive existential sentence with constants admits a solution.
		Hereafter, we omit the term `with constants', considering all positive
		existential sentences to permit them. We delimit systems of equations a
		positive existential sentence with the logical operator \(\vee\); that
		is we denote \(\mathcal F\) as \(\mathcal E_1 \vee \cdots \vee \mathcal
		E_k\), where each \(\mathcal E_i\) is a system of equations.
	\end{dfn}

	It is not difficult to show that on the level of decidability, the
	positive existential theory and the Diophantine problem are equivalent. One
	direction follows as systems of equations are just a special case of
	positive existential sentences, and the converse follows as one can decide
	a finite disjunction by deciding each of the statements within it individually.

	\begin{lem}
		Let \(G\) be a finitely generated group. Then the positive existential
		theory in \(G\) is decidable if and only if the Diophantine problem is
		decidable.
	\end{lem}

	A useful observation (that we will not be exploiting) can be seen using the
	fact that \((A \vee B) \wedge C = (A \wedge C) \vee (B \wedge C)\). This
	implies that any sentence made up of equations using \(\wedge\) and
	\(\vee\) is equivalent (in that it admits exactly the same solutions) to a
	positive existential sentence.

\subsection{Definable sets}

	Much of the proof of undecidability involves showing that certain subsets of
	\(G^n\) can be described using systems of equations, or positive
	existential sentences. We formalise this notion below.
    \begin{dfn}
		Let \(G\) be a finitely generated group and let \(\mathcal E\) be a system
		of equations in \(G\). Let \((X_1, \ldots, X_n)\) be 
		an ordered subset of the variables in
		\(\mathcal E\). We define the \emph{set of solutions} to
        \((X_1, \ldots, X_n)\) in \(\mathcal E\) to be set
        \[
        \{(g_1, \ldots, g_n) \mid \textrm{ there exist } h_1, \ldots, h_m
                \in G \textrm{ such that } (g_1, \ldots, g_n, h_1, \ldots, h_m)
		\text{ is a solution to } \mathcal E\}.
		\]
		Let \(D \subseteq G^n\) for some \(n \in \Z_{> 0}\). We say that \(D\) is
		\emph{equationally definable} in \(G\) if there exists a system
		\(\mathcal E\) of equations in \(G\), and an ordered subset \((X_1,
		\ldots, X_n)\) of the variables in \(\mathcal E\), such that \(D\) is
		the set of solutions to \((X_1, \ldots, X_n)\) in \(\mathcal E\). Sets
		\emph{definable in the positive existential theory} of \(G\) are
		defined analogously.
	\end{dfn}

	An useful fact about the sets definable in the positive existential
	theory of a group that is not true about the class of equationally
	definable sets, is that it is closed under finite unions.

	If \(E \subseteq G^n\) is equationally definable in \(G\), we will
	frequently define a system of equations of the form \(((X_1, \ldots, X_n)
	\in E) \wedge \cdots\), where \((X_1, \ldots, X_n) \in E\) is taken to mean
	any system of equations in \(G\) where the solution set to \((X_1, \ldots,
	X_n)\) is \(E\). We will also do this analogously for sets definable in the
	positive existential theory.

\subsection{Equations in the integers with divisibility constraints}
Unlike most proofs of undecidability of Diophantine problems, we do not
provide a direct reduction to Hilbert's tenth problem over the ring of integers.
Instead, we encode another undecidable number theoretic problem, due to Denef.
Let \(p \in \Z_{> 0}\) be prime. Define the binary relation \(|^p\) on
\(\Z\) by \(x |^p y\) if there exists \(r \in \Z_{\geq 0}\) such that
\(y = \pm x p^r\). We use \(|\) to denote standard divisibility of integers. A result
of Denef states that the Diophantine problem in the signature \((\Z, +, |, |^p)\)
is undecidable \cite[Corollary on page 132]{Denef}. That is, it is undecidable given as input a system \(\mathcal E\) of equations in the
group \(\Z\) together with finitely many statements of the form \(X|Y\) or \(X|^p Y\),
for variables \(X\) and \(Y\),
whether \(\mathcal E\) admits a solution such that for all statements \(X|Y\) (resp.
\(X|^p Y\)) the solution \((x, y)\) to \((X, Y)\) satisfies \(x | y\) (resp. \(x |^p y\)).

\section{Proof of main result}

We begin by working in the solvable Baumslag-Solitar groups. We will later
reduce various questions about \(G_p\) to the solvable Baumslag-Solitar groups.
As with many proofs of undecidability of Diophantine problems in groups,
centralisers provide a useful source of equationally definable sets.

\begin{lem}
	\label{BS-cents}
	Let \(p \in \Z\) be such that \(p > 1\). In \(\operatorname{BS}(1, p) =
	\langle a, b \mid bab^{-1} = a^p \rangle\), we have:
	\begin{enumerate}
		\item \(C_{\operatorname{BS}(1, p)}(a) = \langle b^{-k} a b^k \; (k \in \Z) \rangle\);
			\label{item:BS-cent-a}
		\item \(C_{\operatorname{BS}(1, p)}(b^i) = \langle b \rangle\) for all
				\(i \in \Z \setminus \{0\}\);
			\label{item:BS-cent-b}
		\item \(C_{\operatorname{BS}(1, p)}(ab^i) = \langle ab^i \rangle\) for all
				\(i \in \Z \setminus \{0\}\).
			\label{item:BS-cent-abi}
	\end{enumerate}
\end{lem}

\begin{proof}
	By \cite[Theorem 3.2]{BS-fixed}, every fixed point subgroup of a non-trivial
	automorphism of \(\BS(1, p)\) is either infinite cyclic or contained in
	\(\langle \langle a \rangle \rangle\) (note that the statement does not
	explicitly state this, but the proof does indeed show it). Since
	centralisers are just fixed point subgroups of inner automorphisms, it
	follows that each of the centralisers we are considering are either
	infinite cyclic or contained in \(\langle \langle a \rangle \rangle\). 

	\eqref{item:BS-cent-a}. In this case, it is not difficult to show that
	every conjugate of \(a\) by a power of \(b\) centralises \(a\). Thus
	we must have that \(C_{\BS(1, p)}(a) = \langle \langle a \rangle \rangle\),
	which is the desired set.

	\eqref{item:BS-cent-b}. Note that \(b \notin \langle \langle a \rangle
	\rangle\), but \(b \in C_{\BS(1, p)}(b^i)\), and so \(C_{\BS(1, p)}(b^i)\)
	must be a cyclic subgroup containing \(b\).  When quotienting \(\BS(1, p)\)
	onto \(\Z\), \(b\) gets mapped to a generator, which is never a proper
	power in \(\Z\), and so \(b\) cannot be a proper power in \(\BS(1, p)\).
	Thus \(C_{\BS(1, p)}(b) = \langle b \rangle\).

	\eqref{item:BS-cent-abi}: Note that \(ab^i \in C_{\BS(1,
	p)}(ab^i)\), but \(ab^i \notin \langle \langle a \rangle \rangle\), and so
	\(C_{\BS(1, p)}(ab^i) \not \subseteq \langle \langle a \rangle \rangle\). 
	It thus follows that \(\langle ab^i \rangle\) is a cyclic subgroup of
	\(C_{\BS(1, p)}(ab^i)\). Since \(\langle ab^i \rangle\) is a subgroup of
	\(C_{\BS(1, p)}(ab^i)\), in order to show that these sets coincide, it
	suffices to show that \(ab^i\) is not a proper power. If \(ab^i\) is a
	\(k\)th power of some element \(g\), then after projecting onto \(\Z\), it
	follows that \(b^i\) is also a \(k\)th power. It thus follows that \(k |
	i\), and \(g = ub^r\), for some \(u \in \langle \langle a \rangle \rangle\)
	and \(r \in \Z\), where \(rk = i\). Viewing \(\langle \langle a \rangle
	\rangle\) as \(\Z[1/p]\), we can view \(ub^r\) as \((xp^{-q}, r)\), where
	\(xp^{-q} \in \Z[1/p]\) is a fraction in reduced terms. Then
	\(xp^{-q} + xp^{r - q} + \cdots + xp^{(k - 1)r - q} = 1\). We show
	that \(k = 1\); that is, \(ab^i\) is not a proper power.

	Case 1: \(r \geq 0\), then \(x(1 + p^r + \cdots + p^{(k - 1)r}) = p^q\).
	Reducing mod \(p\) gives \(x \equiv 0 \mod p\). Since \(x p^{-q}\) was a
	fraction in reduced terms, it follows that \(q = 0\). Thus
	\(x(1+ p^r + \cdots + p^{(k - 1)r}) = 1\), which implies that
	\(x = 1\) and \(k = 1\), as required. 

	Case 2: \(r < 0\). We still have \(x(1 + p^r + \cdots + p^{(k - 1)r}) =
	p^q\), but the left hand side contains non-integers (namely the negative
	powers of \(p\)). However, multiplying through by \(p^{-(k - 1)r}\) gives
	\(x(-p^{(k - 1) r} + p^{-(k - 2)r} + \cdots + 1) = p^q\). Again,
	reducing mod \(p\) gives that \(x \equiv 0 \mod p\), and so
	\(q = 0\). It thus follows that \(x(1 + p^{-r} + \cdots + p^{-(k - 1)r})
	= 1\), and so \(x = 1\) and \(k = 1\). Again, we have that being a
	\(k\)th power is not a proper power, as required.

	We have thus shown that \(ab^i\) is not a proper power, and so
	\(\langle ab^i \rangle\) is the maximal cyclic subgroup containing
	it, and therefore must equal \(C_{\BS(1, p)}(ab^i)\).
\end{proof}

We next compute transversals of the two natural copies of \(\Z\) sitting inside
a solvable Baumslag-Solitar group. As \(G_p\) is an HNN-extension of \(\BS(1, p)\)
that identifies these subgroups, this will allow us to compute a normal form
for an important subgroup of \(G_p\). 

\begin{lem}
	\label{transversals}
	Let \(p \in \Z\) be such that \(p > 1\).
	Let \(\BS(1, p) = \langle a, b \mid bab^{-1} = a^p\rangle\). Define
	\begin{align*}
	T_A & =    
		\{ b^{-i} a^j b^k \mid i,k \in \Z_{> 0}, \; j \in \{0, \ldots, p^i -
		1\}, \; j \not\in p\Z \}  \\
& \cup \{ b^k \mid k \in \Z_{\geq 0} \} \\ 
& \cup \{ b^{-i} a^j \mid i \in \Z_{\geq 0}, \; j \in \{0, \ldots, p^i - 1\} \}, \\ 
	T_B & = \{ a^j  b^k\mid k \in \Z_{> 0}, j \in \Z \setminus p\Z\}
	\cup \{a^j \mid j \in \Z\}.
	\end{align*}
	Then \(T_A\) is a right transversal for \(\langle a \rangle\) in
	\(\BS(1, p)\) and \(T_B\) is a right transversal for \(\langle b \rangle\)
	in \(\BS(1, p)\).
\end{lem}

\begin{proof}
	The proof of \cite[Lemma
	3.1]{almost-1rel-DP} shows that \(T_A\) is indeed a right transversal for
	\(\langle a \rangle\) in \(\BS(1, p)\).
	We thus must show that \(T_B\) is a right transversal for \(\langle b
	\rangle\) inside \(\BS(1, p)\). By \cite[Section~2]{BurilloElder}, \(\BS(1,
	p)\) admits the normal form
		\[
				NF = \{ b^{-i} a^j b^k \mid i, k \in \Z_{> 0}, j \in \Z \setminus p\Z \} 
				\cup \{b^{-i} a^j b^k \mid i, k \in \Z_{\geq 0}, ik = 0, j \in \Z\}.
		\]
		We first show that \(T_B\) contains a right transversal for
				\(\langle b \rangle\) in \(\BS(1, p)\). Let \(g \in \BS(1, p)\). If
				the normal form for \(g\) belongs to the first set in the union defining
				\(NF\), then \(g = b^{-i} a^j b^k\), where \(i, k \in \Z_{> 0}\) and
				\(j \in \Z \setminus p\Z\). In this case, \(a^j b^k \in T_B\) and so
				\(g \in \langle b \rangle T_B\). Otherwise, \(g = b^{-i} a^j b^k\), for
				some \(i, k \in \Z_{\geq 0}\) and \(j \in \Z\) such that \(i = 0\) or
				\(k = 0\). If \(k = 0\), then \(a^j \in T_B\) and so \(g \in \langle
				b \rangle T_B\). If \(i = 0\), then \(g = a^j b^k\). If \(j = 0\) then
				\(g \in \langle b \rangle\) and there is nothing to prove, so suppose
				\(j \neq 0\). We can thus write 
				\(j = p^r j'\) for some \(r \in \Z_{\geq 0}\) and \(j' \in \Z \setminus p\Z\).
				Thus \[
						g = a^j b^k = a^{p^rj'} b^k =
						\begin{cases}
							b^k a^{p^{r - k} j'} & r \geq k \\
							b^r a^{j'} b^{k - r} & k \geq r.
						\end{cases}
					\]
				Both cases lie in \(\langle b \rangle T_B\), and so \(T_B\) contains
				a right transversal for \(\langle b \rangle\) in \(\BS(1, p)\).
		
				It thus
		remains to show that distinct elements of \(T_B\) lie in distinct right
		cosets of \(\langle b \rangle\). Let \(u, v \in T_B\) be such that
		\(\langle b \rangle u = \langle b \rangle v\). We consider the three
		cases (up to symmetry) arising from whether \(u\) and \(v\) lie in the
		first or second sets in the union defining \(T_B\). 

		Case 1: \(u = a^j\) and \(v = a^{j'}\) for some \(j, j' \in \Z\). Then
		\(uv^{-1} \in \langle b \rangle\) implies that \(a^{j - j'} = b^k\)
		for some \(k \in \Z\). But both \(a^{j - j'}\) and \(b^k\) belong to
		the normal form \(NF\), and so must be equal as words. It
		follows that \(j - j' = k = 0\), and so \(u = v\), as required.

		Case 2: \(u = a^j b^k\) for some \(k \in \Z_{> 0}\) and \(j \in \Z
		\setminus p \Z\) and \(v = a^{j'}\) for some \(j' \in \Z\). Then
		\(uv^{-1} \in \langle b \rangle\) implies that
		\(a^{j} b^k  a^{-j'} = b^q\) for some \(q \in \Z\). 
		Note that \(a^{j} b^k a^{-j'} = a^{j - p^k j'} b^k\). This is now
		in the normal form \(NF\), and thus must equal \(b^q\) as a word. It follows
		that \(j - p^k j' = 0\). Since \(j \in \Z \setminus p\Z\), this
		implies that \(k = 0\), a contradiction. This case therefore does
		not occur.

		Case 3: \(u = a^j b^k\) and \(v = a^{j'} b^{k'}\) for some
		\(k, k' \in \Z_{> 0}\) and \(j, j' \in \Z \setminus p\Z\). Then
		\(uv^{-1} \in \langle b \rangle\) implies that
		\( a^j b^{k - k'} a^{-j'} = b^q\) for some \(q \in \Z\). We consider
		the subcases \(k - k' > 0\), \(k - k' < 0\) and \(k - k' = 0\)
		separately.

		Subcase 3A: \(k - k' > 0\). Then \(a^j b^{k - k'} a^{-j'} = 
		a^{j -j' p^{k - k'}} b^{k - k'}\), which is now in the normal form \(NF\).
		It thus equals \(b^q\) as a word, and so \(j - j' p^{k - k'}
		= 0\). As with Case 2, \(j \notin p\Z\), and so \(k - k' = 0\),
		a contradiction to our subcase assumption. 

		Subcase 3B: \(k - k' < 0\). Then \(a^j b^{k - k'} a^{-j'}
		= b^{k - k'} a^{j p^{k' - k} - j'}\), which is in normal form,
		and so must equal \(b^q\) as a word. Thus \(j p^{k' - k} -j' = 0\),
		and using the fact that \(j' \notin p\Z\), this implies that
		\(k - k' = 0\), contradicting our subcase assumption.

		Subcase 3C: \(k - k' = 0\). Then \(a^j b^{k - k'} a^{-j'}
		= a^{j - j'}\), which is in normal form, and thus equals 
		\(b^q\) as a word. It follows that \(j - j' = 0\), and so
		\(u = v\), as required.
\end{proof}

The group \(G_p\) can be obtained by taking a semidirect product of
\(\Z\)-indexed copies of \(\BS(1, p)\) joined with \(\Z\)-indexed free products
with amalgamation (\(H_p\)) with \(\Z\) (\cref{inf-pres}). However, much of our
work occurs within a subgroup generated by three of the generators of this
group. We thus require a normal form for this subgroup, which allows us to
prove that specific systems of equations have the sets of solutions that we
claim. Note that after restricting to a Baumslag-Solitar subgroup of \(L_2\),
this normal form does not in fact coincide with the normal form \(NF\) for
Baumslag-Solitar groups that we saw in the proof of \cref{transversals}.

\begin{lem}
		\label{L_2-norm-form}
		Let \(p \in \Z\) be such that \(p > 1\). Let \(L_2 = \langle a_{-1}, a_0, a_1 \mid
		a_0 a_{-1} a_0^{-1} = a_{-1}^p, a_1 a_0 a_1^{-1} = a_0^p \rangle\). Define
	\begin{align*}
	T_A & =    
		\{ a_1^{-i} a_0^j a_1^k \mid i,k \in \Z_{> 0}, \; j \in \{0, \ldots, p^i -
		1\}, \; j \not\in p\Z \}  \\
& \cup \{ a_1^k \mid k \in \Z_{\geq 0} \} \\ 
& \cup \{ a_1^{-i} a_0^j \mid i \in \Z_{\geq 0}, \; j \in \{0, \ldots, p^i - 1\} \}, \\ 
	T_B & = \{ a_{-1}^j  a_0^k\mid k \in \Z_{> 0}, j \in \Z \setminus p\Z\}
	\cup \{a_{-1}^j \mid j \in \Z\}.
	\end{align*}
	Then
	\[
		\{a_0^k x_1 y_1 \cdots x_n y_n \mid k \in \Z, n \in \Z_{> 0}, x_1 \in T_A,
				x_i \in T_A \setminus \{1\} \text{ for } i \neq 1, y_n \in T_B, y_{i} \in
		T_B \setminus \{1\} \text{ for } j \neq n\}
	\]
	is a normal form for \(L_2\).
	Moreover, the normal form for any element of \(\langle a_{-1}, a_0
	\rangle\) is a word over \(\{a_{-1}, a_0\}\).
\end{lem}

\begin{proof}
	We first note that \(L_2\) is a free product of two solvable
	Baumslag-Solitar groups \(\langle a_{-1}, a_0 \mid a_0 a_{-1} a_0^{-1} =
	a_{-1}^p \rangle\) and \(\langle a_0, a_1 \mid a_1 a_0 a_1^{-1} = a_0^p
	\rangle\), amalgamated across the subgroup \(\langle a_0 \rangle\). By
	\cref{transversals}, \(T_B\) is a right transversal for \(\langle a_0
	\rangle\) inside \(\langle a_{-1}, a_0 \mid a_0 a_{-1} a_0^{-1} = a_{-1}^p
	\rangle\) and \(T_A\) is a right transversal for \(\langle a_0 \rangle\)
	inside \(\langle a_0, a_1 \mid a_1 a_0 a_1^{-1} = a_0^p \rangle\). The
	result now follows by the normal form for amalgamated free products; see for
	example \cite[Theorem 11.3]{BogopolskiBook}.

	For the second claim, note that as \(T_B\) is a right transversal for
	\(\langle a_0 \rangle\) in \(\langle a_{-1}, a_0 \rangle\), every element
	in \(\langle a_{-1}, a_0 \rangle\) can be expressed in the form
	\(a_0^k y\) for some \(y \in T_B\). Since this is in normal
	form, this is the unique normal form for this word.
\end{proof}

We next consider an infinitely generated group \(H_p\). As we will see later in
\cref{inf-pres}, \(G_p\) is isomorphic to \(H_p \rtimes \Z\). The next two
lemmas deal with various technical facts about \(H_p\) that are easier to
consider inside \(H_p\), rather than working directly with \(G_p\).

\begin{lem}
	\label{not-conj}
	Let \(p \in \Z\) be such that \(p > 1\) and
	\[
	 H_p = 
			\langle a_i \; (i \in \Z) \mid a_{i + 1} a_i
	a_{i + 1}^{-1} = a_i^p \; (i \in \Z)\rangle.
	\]
	If \(i \neq j\), and \(k \neq 0\), then \(a_i^k\) is not conjugate to any
	power of \(a_j\).
\end{lem}

\begin{proof}
	We can assume without loss of generality that \(i < j\). We consider
	separately the cases when \(j = i + 1\) and \(j > i + 1\).

	Case 1: \(j = i + 1\). We can assume without loss of generality that \(i =
	1\) and so \(j = 2\). Suppose for contradiction that \(a_1^k\) is conjugate
	to \(a_2^\ell\), where \(k, \ell \in \Z \setminus \{0\}\). Note that
	\(H_p\) splits as an amalgamated free product:
	\[
		H_p \cong \langle a_i \; (i \leq 1) \mid a_{i + 1} a_i a_{i + 1}^{-1} = a_i^p 
		\; (i < 1) \rangle \ast_{\Phi} 
		\langle c_i \; (i \geq 1) \mid c_{i + 1} c_i c_{i + 1}^{-1} = c_i^p
		\; (i \geq 1) \rangle,
	\]
	where \(\Phi \colon \langle a_1 \rangle \to \langle c_1 \rangle\) is
	defined by \(\Phi(a_1) = c_1\). Let \(K_1\) denoted the first group in this
	amalgamated free product and \(K_2\) denoted the second. Note that
	\(a_2^\ell\) is cyclically reduced in the sense of \cite[Theorem
	4.6]{MagnusKarrassSolitar}, and so \cite[Theorem
	4.6(i)]{MagnusKarrassSolitar} tells us that there is a sequence \(x_0 =
	a_1^k, x_1, \ldots, x_n = a_2^\ell\), where \(x_i \in \langle a_1 \rangle\)
	for all \(i \in \{1, \ldots, n - 1\}\) and \(x_i\) is conjugate in either
	\(K_1\) or \(K_2\) to \(x_{i + 1}\) for all \(i\). Since \(x_{n - 1} \in
	\langle a_1 \rangle\), and is not \(1\), by changing \(k\) if necessary, we
	can replace \(x_0\) with \(x_{n - 1}\), we can assume that \(n = 1\). Note
	that this new \(k\) cannot be \(0\) as \(x_{n - 1}\) is conjugate to
	\(x_0\) which was non-trivial. In particular, \(a_1^k\) is conjugate in
	either \(K_1\) or \(K_2\) to \(a_2^k\). Given \(a_2^k \notin K_1\), it must
	be that \(a_1\) and \(a_2\) are conjugate in \(K_2\).
	Next note that \(K_2\) splits as an amalgamated free product:
	\[
		K_2 \cong 
		\langle a_1, a_2 \mid a_2 a_1 a_2^{-1} = a_1^p \rangle \ast_{\Psi}
		\langle d_i \; (i \geq 2) \mid d_{i + 1} d_i d_{i + 1}^{-1} = d_i^p 
		\; (i \geq 2) \rangle
	\]
	where \(\Psi \colon \langle d_0 \rangle \to \langle a_0 \rangle\) is
	defined by \(\Psi(d_0) = a_0\). Using the same argument we used to show
	\(a_1^k\) and \(a_2^\ell\) were conjugate in \(K_2\) to apply \cite[Theorem
	4.6(i)]{MagnusKarrassSolitar} again to show that \(a_1^k\) and \(a_2^\ell\)
	are conjugate in \(\langle a_1, a_2 \mid a_2 a_1 a_2^{-1} = a_1^p \rangle\)
	(by again potentially changing \(\ell\) to some other non-zero integer).
	This is just \(\BS(1, p)\), and we have that \(\langle \langle a_1 \rangle
	\rangle \cap \langle a_2 \rangle = \{1\}\), and so \(a_1^k\) is not
	conjugate to \(a_2^k\), a contradiction.

	Case 2: \(j > i + 1\). Again, we can assume without loss of generality that
	\(i = 0\) and so \(j > 1\). The group \(H_p\)
	can be expressed as the amalgamated free product:
	\[
		H_p \cong \langle a_i \; (i \leq j - 1) \mid a_{i + 1} a_i a_{i + 1}^{-1} = a_i^p 
		\; (i < j - 1) \rangle \ast_{\Phi} 
		\langle c_i \; (i \geq j - 1) \mid c_{i + 1} c_i c_{i + 1}^{-1} = c_i^p
		\; (i \geq j - 1) \rangle,
	\]
	where \(\Phi \colon \langle a_{j - 1} \rangle \to \langle c_{j -1}
	\rangle\) is defined by \(\Phi(a_{j - 1}) = c_{j - 1}\). Call the first
	group in the above amalgamated free product \(K_1\) and the latter \(K_2\).
	Let \(\ell \in \Z\).  Noting that \(a_j^\ell\) is cyclically reduced in the
	sense of \cite[Theorem 4.6]{MagnusKarrassSolitar} and by Case 1
	\(a_j^\ell\) is not conjugate to any element in \(\langle a_{j - 1}
	\rangle\), we can thus apply \cite[Theorem 4.6(ii)]{MagnusKarrassSolitar}
	to conclude that all conjugates of \(a_j^\ell\) in \(H_p\) lie in \(K_2\).
	In particular, \(a_0^k\) is not conjugate to \(a_j^\ell\).
\end{proof}

\begin{lem}
		\label{conj-into-ab}
		Let \(H_p = \langle a_i \; (i \in \Z) \mid a_{i + 1} a_i a_{i + 1}^{-1} 
		= a_i^p \; (i \in \Z)\rangle\). If \(h \in \langle a_2^i a_1 a_2^{-i} \;
		(i \in \Z) \rangle\) is such that \(h a_0 h^{-1} \in \langle a_0, a_1
		\rangle\), then \(h \in \langle a_1 \rangle\).
\end{lem}

\begin{proof}
	It suffices to work in the subgroup \(K = \langle a_0, a_1, a_2 \rangle\) of
	\(H_p\). It can be noted that this sits inside \(H_p\) as a free product
	with amalgamation (as discussed in the proof of \cref{cents}), and thus
	\(K\) admits the presentation
	\[
		\langle a_0, a_1, a_2 \mid a_1 a_0 a_1^{-1} = a_0^p, a_2 a_1 a_2^{-1}
		= a_1^p \rangle.
	\]
	This is isomorphic (albeit replacing \(a_i\) with \(a_{i + 1}\)) to the
	group discussed in \cref{L_2-norm-form}, and thus we can use this normal
	form for \(K\). Let \(h \in \langle a_2^i a_1 a_2^{-i} \; (i \in \Z)
	\rangle\). Since this lies in the Baumslag-Solitar subgroup
	\(\langle a_1, a_2\rangle\), we can write \(h\) using the normal
	form
	\[
			NF = \{ a_2^{-i} a_1^j a_2^k \mid i, k \in \Z_{> 0}, j \in \Z \setminus p\Z \} 
			\cup \{a_2^{-i} a_1^j a_2^k \mid i, k \in \Z_{\geq 0}, ik = 0, j \in \Z\}
	\]
	for \(\BS(1, p)\) from \cite[Section~2]{BurilloElder}.

	Case 1: \(h = a_2^{-i} a_1^j a_2^k\) for some \(i, k \in \Z_{> 0}\) and \(j
	\in \Z \setminus p \Z\). Since the relations in \(\langle a_1, a_2 \rangle
	\cong \BS(1, p)\) preserve the exponent-sum of \(a_2\), and elements of
	\(\langle a_2^i a_1 a_2^{-i} \; (i \in \Z) \rangle\) have an \(a_2\)-exponent-sum
	of \(0\), so must \(h\). It follows that \(k = i\), and so
	\(h = a_2^{-i} a_1^j a_2^i\). Then
	\[
		ha_0 h^{-1} = a_2^{-i} a_1^j a_2^i a_0 a_2^{-i} a_1^{-j} a_2^i.
	\]
	Let \(T_A\) and \(T_B\) are defined as in the
	statement of \cref{L_2-norm-form} (except with \(a_i\) replaced by \(a_{i + 1}\) to match
	our presentation for \(K\)). We convert the above word for \(ha_0 h^{-1}\)
	into the normal form from \cref{L_2-norm-form}.  Write \(j = q p^i +  j'\), where \(q \in
	\Z\) and \(j' \in \{0, \ldots, p^i - 1\}\). Then 
	\begin{align*}
		ha_0h^{-1}
		& = a_2^{-i} a_1^j a_2^i a_0 a_2^{-i} a_1^{-j} a_2^i & &
		\\
		& = a_2^{-i} a_1^{j' + q p^i} a_2^i a_0 a_2^{-i} a_1^{(p^i -j') - (q + 1)p^i} a_2^i & & {\color{black} \mbox{(using the fact that \(j = j' + qp^i\)) }} \\
		& = a_1^q a_2^{-i} a_1^{j'} a_2^i a_0 a_1^{-(q + 1)} a_2^{-i} a_1^{p^i -j'} a_2^i & & 
		{\color{black} \mbox{(since \(a_2^{-i} a_1^{rp^i} = a_1^r a_2^{-i}\) for all \(r \in \Z\))}} \\
		& = a_1^q a_2^{-i} a_1^{j'} a_2^i a_1^{-(q + 1)} a_0^{p^{(q + 1)}} a_2^{-i} a_1^{p^i - j'} a_2^i 
		&& {\color{black} \mbox{(since \(a_0 a_1^{-(q + 1)} = a_1^{-(q + 1)} a_0^{p^{q + 1}}\))}} \\
		& = a_1^q a_2^{-i} a_1^{j'} a_1^{-(q + 1)p^i} a_2^i a_0^{p^{(q + 1)}} a_2^{-i} a_1^{p^i - j'} a_2^i 
		&& {\color{black} \mbox{(since \(a_2^i a_1^{-(q + 1)} = a_1^{-(q + 1)p^i} a_2^i\))}} \\
		& = a_1^{-1} a_2^{-i} a_1^{j'} a_2^i a_0^{p^{(q + 1)}} a_2^{-i} a_1^{p^i - j'} a_2^i 
		&& {\color{black} \mbox{(since \(a_2^{-i} a_1^{-(q + 1)p^i} = a_1^{-(q + 1)} a_2^{-i}\)).}} 
	\end{align*}
	We now have that \(a_1^{-1} \in \langle a_1 \rangle\), \(a_2^{-i} a_1^{j'}
	a_2^i \in T_A\), \( a_2^{-i} a_1^{p^i - j'} a_2^i \in T_A\) and
	\(a_0^{p^{(q + 1)}} \in T_B\). The above word is thus in the normal form
	from \cref{L_2-norm-form}, and thus can only represent an element of
	\(\langle a_0, a_1 \rangle\) if it does not contain an occurrence of
	\(a_2\). This forces \(i = 0\), which forces \(h = a_1^j\), as required.

	Case 2: \(h = a_2^{-i} a_1^j a_2^k\) for some \(i, k \in \Z_{\geq 0}\) with
	\(ik = 0\) and \(j \in \Z\). By the argument from the first sentence of
	Case 1, we must have that \(i = k\), and therefore \(i = k = 0\). It
	follows that \(h = a_1^j\), as required.
\end{proof}

We now begin to work in \(G_p\) itself. We start with a description of
\(G_p\) as a semidirect product of an infinite presentation with \(\Z\).

\begin{lem}
	\label{inf-pres}
	Let \(G_p = \langle a, b, t \mid bab^{-1} = a^p, tat^{-1} = b\rangle\), where
	\(p \in \Z\) is such that \(p > 1\). Let
	\[
	 H_p = 
			\langle a_i \; (i \in \Z) \mid a_{i + 1} a_i
	a_{i + 1}^{-1} = a_i^p \; (i \in \Z)\rangle.
	\]
	Then \(G \cong H_p \rtimes_\varphi \Z\), where \(\varphi\) is defined by
	\(\varphi(a_i) = a_{i + 1}\) for all \(i \in \Z\).
\end{lem}

\begin{proof}
	 We begin by viewing \(G_p\) with the presentation \(\langle a, t \mid
	 (tat^{-1}) a (tat^{-1})^{-1} = a^p \rangle\). We then replace \(a\) with
	 \(a_0\), and introduce redundant generators \(a_i = t^i a t^{-i}\) for all
	 \(i \in \Z\), and redundant relators \(a_{i + 1} a_i a_{i + 1}^{-1} =
	 a_i^p\) for all \(i \in \Z \setminus \{0\}\), to obtain the following
	 presentation for \(G_p\):
	\[
		\langle t, a_i \; (i \in \Z) \mid a_{i + 1} a_i
		a_{i + 1}^{-1} = a_i^p, t a_i t^{-1} = a_{i + 1} \; (i \in \Z)\rangle
\]
	 We finally notice that this presentation is of the form \(H_p \rtimes_\varphi \Z\).
\end{proof}

The following technical lemma allows us to compute a centraliser in \(G_p\).
\begin{lem}
	\label{pre-cents}
	Let \(G_p = \langle a, b, t \mid bab^{-1} = a^p, tat^{-1} = b\rangle\),
	where \(p \in \Z\) is such that \(p > 1\). Let \(T_B\) denote
	the right transversal for \(\langle a \rangle\) inside the Baumslag-Solitar subgroup \(\langle t^{-1} a t , a\rangle\),
	as defined in \cref{L_2-norm-form}. Let \(q \in \Z \setminus
	\{0\}\).  If \(y a^q = a^s y\) for some \(y \in T_B\) and \(s \in \Z\),
	then \(y = 1\) and \(s = q\).
\end{lem}

\begin{proof}
	First note that since \(y\) and \(a^q\) both lie in the Baumslag-Solitar
	subgroup \(\langle t^{-1} a t, a \rangle\), it suffices to work in this
	subgroup. As with \(H_p\), write \(a_0 = a\) and \(a_{-1} = t^{-1} a t\).
	By applying the homomorphism from \(\langle a_{-1}, a_0 \rangle\)
	onto \(\Z\) by mapping \(a_0\) to the generator \(c\) and mapping
	\(a_{-1}\) to the identity, we can see that
	\(ya_0^q = a_0^s y\) implies that \(c^q = c^s\), and thus
	\(s = q\). We thus have that \(y a_0^q = a_0^q y\), and so
	\(y \in C_{\langle a_{-1}, a_0 \rangle}(a_0^q)\),
	and thus \cref{BS-cents}\eqref{item:BS-cent-b} tells us that
	\(y \in \langle a_0 \rangle\). But \(T_B \cap \langle a_0 \rangle
	= \{1\}\), and so \(y = 1\).
\end{proof}

We compute the centraliser of \(a\) in \(G_p\). This allows us to later show that
	 the Baumslag-Solitar subgroup \(\langle a, b \rangle\) is equationally definable
	 in \(G_p\).
\begin{lem}
		\label{cents}
		Let \(G_p = \langle a, b, t \mid bab^{-1} = a^p, tat^{-1} = b\rangle\),
		where \(p \in \Z\) is such that \(p > 1\).  Then \(C_{G_p}(a) = \langle
		b^{-i} a b^i \; (i  \in \Z_{\geq 0}) \rangle\).
\end{lem}

\begin{proof}
		As with \cref{pre-cents}, we will view \(G_p\) as \(H_p \rtimes_\phi \Z\) from
	\cref{inf-pres}, where \(t\) is the generator for \(\Z\).  We first
	show that \(C_{G_p}(a_0) = C_{H_p}(a_0)\). Let \(h t^k \in C_{G_p}(a_0)\),
	where \(h \in H_p\) and \(k \in \Z\). Then \(a_0 h t^k =_{G_p} ht^k a_0
	=_{G_p} h a_k t^k\), and so \(a_0 =_{H_p} h a_k h^{-1}\). If \(k \neq 0\),
	then \(a_0\) and \(a_k\) are conjugate in \(H_p\), which contradicts
	\cref{not-conj}, and so \(k = 0\).  We have thus shown that \(ht^k = h \in
	H_p\), and so \(C_{G_p}(a_0) \subseteq C_{H_p}(a_0)\). We thus need to
	compute this latter centraliser. First note that we can view \(H_p\) as a
	free product with amalgamation:
	\[
		H_p \cong \langle a_i \; (i \leq 1) \mid a_{i + 1} a_i a_{i + 1}^{-1} = a_i^p 
		\; (i < 1) \rangle \ast_{\Phi} 
		\langle c_i \; (i \geq 1) \mid c_{i + 1} c_i c_{i + 1}^{-1} = c_i^p
		\; (i \geq 1) \rangle,
	\]
	where \(\Phi \colon \langle a_1 \rangle \to \langle c_1 \rangle\) is the
	isomorphism defined by \(a_1 \mapsto c_1\). Call the first group in this
	free product with amalgamation \(K_1\) and the latter \(K_2\).  By
	\cref{not-conj}, \(a_0\) is not conjugate to any element in \(\langle a_1
	\rangle\). Thus \cite[Proposition 2.3]{BarkThesis} tells us that
	\(C_{H_p}(a_0) = C_{K_1}(a_0)\).  Looking inside \(K_1\), we have that
	\(K_1\) can also be realised as a free product with amalgamation:
	\[
		K_1 \cong 
		\langle d_i \; (i \leq -1) \mid d_{i + 1} d_i d_{i + 1}^{-1} 
		\; (i <  -1) \rangle \ast_{\Psi} \langle a_{-1}, a_0, a_1 \mid
		a_0 a_{-1} a_0^{-1} = a_{-1}^p, a_1 a_0 a_1^{-1} = a_0^p \rangle,
	\]
	where \(\Psi \colon \langle d_{-1} \rangle \to \langle a_{-1} \rangle\) is
	the isomorphism defined by \(d_{-1} \mapsto a_{-1}\).  Call the first group
	\(L_1\) and the second \(L_2\). By \cref{not-conj}, \(a_0\) is not
	conjugate to any element of \(\langle a_{-1} \rangle \) in \(H_p\), and so
	we can apply \cite[Proposition 2.3]{BarkThesis} to conclude that
	\(C_{K_1}(a_0) = C_{L_2}(a_0)\). 

	We next show that \(C_{L_2}(a_0) = C_{\langle a_0, a_1 \rangle}(a_0)\). To
	do so, we use the normal form from \cref{L_2-norm-form}.  Let \(a_0^k x_1
	y_1 \cdots x_n y_n\) be a normal form word representing an element \(g \in
	C_{L_2}(a_0)\). Then \(a_0 g\) has normal form \(a_0^{k + 1} x_1 y_1 \cdots
	x_n y_n\). We next compute the normal form for \(g a_0\). We have \(g a_0
	=_{L_2} a_0^k x_1 y_1 \cdots x_n y_n a_0\). Note that since
	\(T_A\) and \(T_B\) are right transversals for \(\langle a_0\rangle\)
	inside \(\langle a_0, a_1 \rangle\) and \(\langle a_{-1}, a_0 \rangle\),
	respectively, we have that for all \(x \in T_A\) and \(y \in T_B\),
	\(x a_0 = a_0^r x'\) and \(y a_0 = a_0^s y'\) for some \(x' \in T_A\),
	\(y' \in T_B\) and \(r, s \in \Z\). We will show the following
	
	We now use (i) and (ii) to prove by induction on \(n\) that if \(a_0^k x_1
	y_1 \cdots x_n y_n = x_1 y_1 \cdots x_n y_n a_0^{k'}\) for some \(k, k' \in
	\Z \setminus \{0\}\), then \(n = 1\) and \(y_1 = 1\). If \(n = 1\), then
	\[
			a_0^{k} x_1 y_1 = x_1 y_1 a_0^{k'} = x_1 a_0^s y_1' =
		a_0^{r} x_1' y_1',
	\]
	for some \(r \in \Z\), \(x_1' \in T_A\) and \(y_1' \in T_B\). As this is in
	normal form, this forces \(y_1 = y_1'\), and so \(y_1 a_0^{k'} = a_0^s
	y_1'\). 
	By \cref{pre-cents}, we have that \(y_1 = 1\), as required.

	Now suppose \(n > 1\) and the result holds for all \(n' < n\). Then,
	(writing \(k' = r_n\))
	\begin{align*}
		a_0^{k} x_1 y_1 \cdots x_n y_n 
		& = x_1 y_1 \cdots x_n y_na_0^{k'} \\
		& = x_1 y_1 \cdots x_n y_na_0^{r_n} \\
		& = x_1 y_1 \cdots x_n a_0^{s_{n}} y_n' \\
		& = x_1 y_1 \cdots a_0^{r_{n - 1}} x_n' y_n' \\
		& = \cdots \\
		& = a_0^{r_{0}} x_1' y_1' \cdots x_n' y_n',
	\end{align*}
	where \(x_1', \ldots, x_n' \in T_A\), \(y_1', \ldots, y_n' \in T_B\) and
	\(s_1,r_1, \ldots, s_{n}, r_n \in \Z\) are defined inductively by \(y_i
	a_0^{r_i} = a_0^{r_i} y_i'  \) and \(x_i a_0^{s_i} = a_0^{r_{i - 1}}
	x_i'\), noting that these are all uniquely determined (and exist) as
	\(T_A\) and \(T_B\) are right transversals of \(\langle a_0 \rangle\) in
	\(\langle a_0, a_1 \rangle\) and \(\langle a_{-1}, a_0 \rangle\),
	respectively. Furthermore, note that if \(x_i \neq 1\), then \(x_i' \neq
	1\), as otherwise we would have \(x_i = a_0^{r_{i - 1} - s_i} \in T_A \cap
	\langle a_0 \rangle = \{1\}\), a contradiction. Similarly, \(y_i \neq 1\)
	implies that \(y_i' \neq 1\). This word is thus is in the normal form from
	\cref{L_2-norm-form}, it must equal \(a_0^{k} x_1 y_1 \cdots x_n y_n\) as a
	word, and so \(x_i = x_i'\) and \(y_i = y_i'\) for all \(i\). In
	particular, \(y_n a_0^{k'} = a_0^s y_n\), which again (i) and (ii) tell us
	implies \(y_n = 1\). This then implies that \(a_0^{k} x_1 y_1 \cdots x_n =
	x_1 y_1 \cdots y_{n - 1} a_0^s x_n\), and so \(a_0^{k} x_1 y_1 \cdots x_{n
	- 1} y_{n - 1} = x_1 y_1 \cdots x_{n - 1} y_{n - 1} a_0^s\). We can now
	apply induction to conclude that \(n - 1 = 1\) and \(y_{n - 1} = 1\). But
	the structure of our normal form for \(g\) states that \(y_i \neq 1\) for
	all \(i < n\), a contradiction. Thus this case of \(n > 1\) does not occur.
	
	 We have now shown that \(g = a_0^k x_1\) for some \(k \in \Z\) and \(x_1
	 \in T_A\), and so \(g \in \langle a_0, a_1 \rangle\), and so
	 \(C_{G_p}(a_0) = C_{\langle a_0, a_1 \rangle}(a_0)\). Since \(\langle a_0,
	 a_1 \rangle\) is isomorphic to \(\operatorname{BS}(1, n)\), we can simply
	 use \cref{BS-cents}\eqref{item:BS-cent-a} to conclude that \(C_{G_p}(a_0)
	 = \langle a_1^{-i} a_0 a_1^{i} \; (i \in \Z_{\geq 0}) \rangle\).
\end{proof}

We can now show that the sets we require are definable in \(G_p\). These sets
include the two divisibility conditions over \(\Z\) that we need to encode
Denef's result on the undecidability of the Diophantine problem in
\((\Z, +, |, |^p)\) \cite{Denef}.

\begin{lem}
	\label{defable}
	The following sets are definable in the positive existential theory of
	\(G_p = \langle a, b, t \mid bab^{-1} = a^p, tat^{-1} = b\rangle\): 
	\begin{enumerate}
			\item \(\langle b \rangle\);
					\label{item:b-defable}
			\item \(\langle a \rangle\);
					\label{item:a-defable}
			\item \(\Mon \langle b \rangle\);
					\label{item:mon-b-defable}
			\item \(\langle a, b \rangle\);
					\label{item:ab-defable}
			\item \(\{(b^i, b^{ij}) \mid i, j \in \Z\}\);
					\label{item:b-div-defable}
			\item \(\{(b^i, b^{i p^r}) \mid i \in \Z, r \in \Z_{\geq 0}\}\);
					\label{item:b-divp-pre-defable}
			\item \(\{(b^i, b^{\pm i p^r}) \mid i \in \Z, r \in \Z_{\geq 0}\}\).
					\label{item:b-divp-defable}
	\end{enumerate}
\end{lem}

\begin{proof}
 	\eqref{item:b-defable}: We first apply \cref{cents} to show that
	\[
			C_{G_p}(b) = C_{G_p}(tat^{-1}) = t C_{G_p}(a) t^{-1}
			= t \langle b^{-i} a b^i \; (i \in \Z_{\geq 0}) \rangle t^{-1}
					= \langle  (tbt^{-1})^{-i} b (tbt^{-1})^i \; (i \in \Z_{\geq 0}) \rangle.
	\]
	We then consider the system \((Xb = bX) \wedge (XaX^{-1} a = a XaX^{-1})\).
	The first equation insists that a solution \(x\) to \(X\) lies in \(\langle
	(tbt^{-1})^{-i} b (tbt^{-1})^i \; (i \in \Z)\rangle\), and the second
	equation insists (again, using \cref{cents}) that \(xax^{-1} \in C_{G_p}(a)
	=  \langle b^{-i} a b^i \; (i \in \Z_{\geq 0}) \rangle\). Using the
	description of \(G_p\) as \(H_p \rtimes_\varphi \Z\) from
	\cref{inf-pres}
	along with
	\cref{conj-into-ab}, the only elements of \(\langle (tbt^{-1})^{-i} b (tbt^{-1})^i \;
	(i \in \Z)\rangle\) that conjugate \(a\) into \(\langle a, b \rangle\) are
	those in \(\langle b \rangle\). Thus every solution to this system of
	equations lies in \(\langle b \rangle\).  Since conjugates of \(a\) by
	powers of \(b\) commute with \(a\), it follows that \(\langle b \rangle\)
	is precisely the set of solutions.

	\eqref{item:a-defable}: Since \(\langle b \rangle\) is equationally
	definable by \eqref{item:b-defable}, and \(t^{-1} b^i t = a^i\) for all \(i
	\in \Z\), it follows that \(\langle a \rangle\) is the set of solutions to
	\(Y\) in the system \((X \in \langle b \rangle) \wedge (t^{-1} X t = Y)\).

	\eqref{item:mon-b-defable}: Note that from \eqref{item:a-defable} and
	\eqref{item:b-defable}, \(\langle a \rangle\) and \(\langle b \rangle\) are
	equationally definable in \(G_p\). Then consider the system \((X \in
	\langle b \rangle) \wedge (Y \in \langle a \rangle) \wedge (X aX^{-1} =
	Y)\). Noting that \(\langle a, b \rangle\) is a copy of \(\BS(1, p)\), we
	can use the fact that \(b^i a b^{-i} \in \langle a \rangle\) if and only if
	\(i \geq 0\), and thus it follows that the set of solutions to \(X\) in
	this system is \(\Mon\langle b \rangle\).

	\eqref{item:ab-defable}: Using \cref{cents} and \eqref{item:b-defable}, 
	we have that \(C_{G_p}(a) = \langle b^{-i} a b^i \; (i \in \Z_{\geq 0}) \rangle\)
	and \(\langle b \rangle\) are both equationally definable in \(G_p\).
	Recall that \(\langle a, b \rangle\) is a solvable Baumslag-Solitar group,
	and is thus the semidirect product of these subgroups. In particular,
	\(\langle a, b \rangle = \langle b^{-i} a b^i \; (i \in \Z_{\geq 0}) \rangle
	\langle b \rangle\). As a product of equationally definable sets,
	\(\langle a, b \rangle\) is thus equationally definable.

	\eqref{item:b-div-defable}: Using \eqref{item:ab-defable}, the solvable
	Baumslag-Solitar subgroup \(B = \langle a, b \rangle\) is equationally
	definable. It thus suffices to show that \(\{(b^i, b^{ij}) \mid i, j \in
	\Z\}\) is equationally definable in \(B\). We have from \cref{BS-cents}
	\eqref{item:BS-cent-b} that \(C_B(b) = \langle b \rangle\). Thus the
	set of solutions to the system of equations in \(B\): \((Xb = bX) \wedge (Y = bX^2)\) is
	\(\{(b^i, b^{2i + 1}) \mid i \in \Z\}\). It follows that
	\(\{b^i \mid i \in 2 \Z + 1\}\) is equationally definable in \(B\). From
	\cref{BS-cents}\eqref{item:BS-cent-abi}, we have that \(C_B(ab^i) = \langle ab^i
	\rangle\) for all \(i \in 2 \Z + 1\). Thus the set of solutions to
	\((Y \in \{b^i \mid i \in 2\Z + 1\}) \wedge (Z (aY) = (aY)Z)\) is
	\[
		\{(b^i, (ab^i)^j) \mid i, j \in \Z, i \text{ odd}\}.
	\]
	We have that for all \(i, j \in \Z\), \((ab^i)^j = a^{1 + p^i + \cdots +
	p^{(j - 1)i}} b^{ij}\). Thus if we add the equations \(Ua = aU\), \(Vb =
	bV\) and \(UV = Z\), (using the fact that \(B\) is a semidirect product of
	\(C_B(a)\) with \(\langle b \rangle\) from
	\cref{BS-cents}\eqref{item:BS-cent-a}), we have that the set of solutions
	to \((Y, Z, U, V)\) in this system will be
	\[
			\{(b^i, a^{1 + p^i + \cdots + p^{(j - 1)i}} b^{ij},  a^{1 + p^i +
			\cdots + p^{(j - 1)i}}, b^{ij}) \mid i, j \in \Z, i \text{ odd}\}.
	\]
	In particular, the set \(\{(b^i, b^{ij}) \mid i, j \in \Z, i \text{ odd}\}\)
	is equationally definable in \(B\). It remains to `remove' the oddness
	of \(i\). To do so, we take \((Y, V)\) from the previous system, and add
	the equations
	\[
		(\bar{Y} b = b \bar{Y}) \wedge
		(\bar{V} b = b \bar{V}) \wedge
		(b \bar{Y}^2 = Y) \wedge
		(b \bar{V}^2 = V).
	\]
	The first two equations insist all solutions to \(\bar{Y}\) and \(\bar{V}\)
	will lie in \(\langle b \rangle\) and the latter two insist that the exponents of
	the solutions to \((Y, V)\) can be obtained from the exponents of the solutions
	to \((\bar{Y}, \bar{V})\) by doubling and then adding \(1\). In other words,
	the set of solutions to \((\bar{Y}, \bar{V}, Y, V)\) is
	\[
			\{(b^{\bar{i}}, b^{\bar{i} j}, b^{2 \bar{i} + 1},
			b^{(2 \bar i + 1)j}) \mid \bar{i}, j \in \Z\}.
	\]
	We have thus shown the desired set is equationally definable in \(B\), and thus it
	is also equationally definable in \(G_p\), since \(B\) is equationally definable
	in \(G_p\).

	\eqref{item:b-divp-pre-defable}: First note that \(\langle a \rangle\) and
	\(\Mon\langle b \rangle\) are equationally definable by \eqref{item:a-defable}
	and \eqref{item:mon-b-defable}. Then consider the system \((X \in \langle a
	\rangle) \wedge (Y \in \Mon\langle b \rangle) \wedge (Z = YXY^{-1}) \). Since
	\(b^r a^i b^{-r} = a^{ip^r}\) for all \(i \in \Z\) and \(r \in \Z_{\geq 0}\) 
	the set of solutions to \((X, Y, Z)\) in this system is
	\[
			\{(a^i, b^r, a^{ip^r}) \mid i \in \Z, r \in \Z_{\geq 0}\}.
	\]
	We have thus shown that \(\{(a^i, a^{ip^r}) \mid i \in \Z, r \in \Z_{\geq 0}\}\)
	is equationally definable in \(G_p\). It thus follows that if
	\((X, Y)\) are variables taking solutions in this set, and we add new
	equations \((X' = tXt^{-1}) \wedge (Y' = t Y t^{-1})\), then the
	solutions to \((X', Y')\) will be the desired set.

	\eqref{item:b-divp-defable}: Note that the system
	\[
			(X, Y) \in \{(b^i, b^{i p^r}) \mid i \in \Z, r \in \Z_{\geq 0}\}) \wedge
			(Z = Y^{-1})
	\]
	has a solutions set \(\{(b^i, b^{i p^r}, b^{-ip^r}) \mid i \in \Z, r \in \Z_{\geq 0}\}\),
	and so \(\{(b^i, b^{-ip^r}) \mid i \in \Z, r \in \Z_{\geq 0}\}\) is equationally
	definable in \(G_p\). It thus follows that the desired set is definable in the
	positive existential theory of \(G_p\), as a union of two equationally definable
	sets in \(G_p\).
\end{proof}

We have now completed the majority of the proof of \cref{main-thm}. It simply
remains to combine some of the results from \cref{defable} to encode
the Diophantine problem in the signature \((\Z, +, |, |^p)\).

\begin{proof}[Proof of \cref{main-thm}]
	As discussed earlier, it suffices to show that the positive existential
	theory in \(G_p\) is undecidable. To do so, we encode \(\Z\) with the
	operations \(+\), \(|\) and \(|^p\) into the positive existential theory of
	\(G\). As done earlier, if we use \(b = tat^{-1}\), the we can apply
	\cref{defable}\eqref{item:b-defable} to show the infinite cyclic subgroup
	\(\langle b \rangle\) is definable in the positive existential theory of
	\(G\). We can thus encode any positive existential existential sentence in \((\Z, +)\) into the positive existential theory of \(G_p\).
	Adding the divisibility constraints \(|\) and \(|^p\) can then be done by
	\cref{defable}\eqref{item:b-div-defable} and \eqref{item:b-divp-defable}.
	Since the positive existential theory in \((\Z, +)\) with the divisibility
	constraints of the form \(|\) and \(|^p\) is undecidable
	\cite[Corollary]{Denef}, the result follows.
\end{proof}

In order to show \cref{main-cor}, we now show that the Diophantine problem
in a group \(G\) reduces to the single equation problem in \(G \ast \Z\). The
idea of reducing systems of equations to single equations has been done before,
such as in \cite{Bogopolski22}, but some of these methods are not computable.

\begin{lem}
	\label{SEP-DP}
	Let \(G\) be a finitely generated group. If \(G \ast \Z\) has a decidable
	single equation problem, then \(G\) has a decidable Diophantine problem.
\end{lem}

\begin{proof}
	Let \(\#\) be the generator for the copy of \(\Z\) in \(H = G \ast \Z\).
	For each \(u \in H\) and \(i \in \Z\), define \(\mathcal{W}_i(u) = \#^{i +
	1} u \#^{-i} u^{-1}\).  We first claim that for all \(u \in H\) that
	\(\mathcal W_i(u)\) is not conjugate to any element in \(G\). To see this,
	we can project onto \(\langle \# \rangle\), and note that \(\#^{i + 1} u
	\#^{-i} u^{-1}\) will be mapped to \(\# \neq 1\), and so \(\mathcal
	W_i(u)\) is not conjugate to any element in \(G\). We can thus apply for
	example \cite[Theorem 2.1]{BarkThesis} to conclude that \(C_H(\mathcal
	W_i(u))\) is cyclic. Since \(\mathcal W_i(u)\) projects onto \(\#\),
	which is not a proper power in \(\langle \# \rangle\), \(\mathcal W_i(u)\)
	cannot be a proper power, and so \(C_H(\mathcal W_i(u)) = \langle
	\mathcal W_i(u) \rangle\).

	We next show that for \(i \neq j\) and \(u, v \in H\) that \(\mathcal
	W_i(u)\) and \(\mathcal W_j(v)\) commute if and only if \(u, v \in \langle
	\# \rangle\).  The backward direction is immediate, so suppose that
	\(\mathcal W_i(u) \) and \(\mathcal W_j(v)\) commute. As we calculated
	above, these elements both generate their own centralisers, and so either
	\(\mathcal W_i(u) = \mathcal W_j(v)\) or \(\mathcal W_i(u) = \mathcal
	W_j(v)^{-1}\). Projecting to \(\langle \# \rangle\) rules out the latter,
	and so \(\mathcal W_i(u) = \mathcal W_j(v)\); that is, \(\#^{i + 1} u
	\#^{-i} u^{-1} = \#^{j + 1} v \#^j v^{-1}\). Then \([\#^i, u] = [\#^j,
	v]\). 

	Write \(u = u_1 \#^{i_1} \cdots u_m \#^{i_m}\) and \(v = v_1 \#^{j_1}
	\cdots v_n \#^{j_n}\), where \(u_1, v_1 \in G, u_2, v_2, \ldots, u_m,
	v_n \in G \setminus \{1\}\), \(i_1, j_1, \ldots, i_{m - 1}, j_{n - 1} \in
	\Z \setminus \{0\}\) and \(i_m, j_n \in \Z\); that is, in the standard
	normal form for free products. Then
	\begin{align*}
		[\#^i, u] & = \#^i u_1 \#^{i_1} \cdots u_m \#^{-i} u_m^{-1}
		\cdots \#^{-i_1} u_1^{-1} \\
		[\#^j, v] & = \#^j v_1 \#^{j_1} \cdots v_n \#^{-j} v_n^{-1}
		\cdots \#^{-j_1} v_1^{-1}. 
	\end{align*}
	These words may not quite be in the normal form from \cref{L_2-norm-form}, since \(u_1\) and
	\(v_1\) can be \(1\). But that can only affect the beginning of the normal
	form words, so we start by comparing the ends. Suppose for a contradiction
	that \(n > 1\). Since these elements
	coincide, we must have that
	\(u_1 = v_1\). Thus these normal form words must have the same alternating
	length; that is, \(m = n\). By comparing these words, we
	obtain \(i = j\), a contradiction. Thus \(n = 1\), and so comparing again
	gives \(m = 1\). We thus have that \(u = u_1 \#^{i_1}\) and
	\(v = v_1 \#^{j_1}\). Then \([\#^i, u] = \#^i u_1 \#^{-i} u_1^{-1}\)
	and \([\#^j, v] = \#^j v_1 \#^{-j} v_1^{-1}\). Again, comparing the
	ends of these words gives \(u_1 = v_1\). It thus follows that
	\(\#^i u_1 \#^{-i} u_1^{-1} = \#^j u_1 \#^{-j} u_1^{-1}\), and so
	\(\#^{i - j} u_1 \#^{-(i - j)} = u_1\); that is, \(u_1 \in
	C_H(\#^{i - j})\). Since \(i - j \neq 0\), we have that
	\(C_H(\#^{i -j}) = C_{\langle \# \rangle}(\#^{i - j}) = \langle \# \rangle\),
	and so \(u_1 \in \langle \# \rangle\). But \(u_1 \in G\), and so
	\(u_1 = v_1 = 1\). We can thus conclude that \(u_1, v_1 \in \langle \#
	\rangle\).

	We have now shown that for all \(i \neq j\) and all \(u, v \in H\) that
	\([\mathcal W_i(u), \mathcal W_j(v)] = 1\) if and only if \(u, v \in
	\langle \# \rangle\). It follows that if \(w_1, \ldots, w_n \in H\), then
	\[
		[\mathcal W_1(w_1), 
		\mathcal W_2([\mathcal W_3(w_2), 
		[\mathcal W_4([\mathcal W_5(w_3)], \cdots )])] \cdots )] = 1
	\]
	if and only if \(w_1, \ldots, w_n \in \langle \# \rangle\). Call
	the left hand side of the above equation \(\mathcal W(w_1, \ldots, w_n)\).
	We generalise the notation \(\mathcal W_i(u)\) and \(\mathcal W(w_1, \ldots, w_n)\)
	to the case where \(u, w_1 \ldots, w_n \in H \ast F(\mathcal X)\) for some finite
	set \(X\).

	So now suppose \(\mathcal E = (w_1 = 1) \wedge \cdots \wedge (w_n = 1)\) is
	a system of equations in \(G\) with variables \(\mathcal X\). We claim that
	\(\mathcal E\) admits a solution in \(G\) if and only if \(\mathcal W(w_1,
	\ldots, w_n) = 1\) admits a solution in \(H\). If \(\mathcal E\) admits a
	solution \(\phi \colon G \ast F(\mathcal X) \to G\), then \(\phi\) will
	extend trivially to a solution to \(\mathcal W(w_1, \ldots, w_n)\). Now
	suppose \(W(w_1, \ldots, w_n)\) admits a solution \(\psi \colon H \ast
	F(\mathcal X) \to H\). The properties of \(\mathcal W\) tells us that this
	implies that \(\psi(w_1), \ldots, \psi(w_n) \in \langle \# \rangle\). Let
	\(r \colon H \to G\) be the retraction defined by \(r(\#) = 1\) and
	\(r(g) = g\) for all \(g \in G\). Then \(r \circ \psi(w_i) = 1\) for
	all \(i \in \{1, \ldots, n\}\), and so \(\psi\) is a solution to
	\(\mathcal E\). We have thus shown that \(\mathcal E\) admits a solution
	in \(G\) if and only if \(\mathcal W(w_1, \ldots, w_n) = 1\) admits a solution
	in \(H\).

	Noting that \(\mathcal W(w_1, \ldots, w_n)\) can be effectively computed
	from \(\mathcal E\), we now have that for every system of equations
	\(\mathcal E\) in \(G\) there is a computable single equation in \(H\) that
	admits a solution in \(H\) if and only if \(\mathcal E\) does in \(G\). We
	can thus use the single equation problem in \(H\) to decide the Diophantine
	problem in \(G\), as required.
\end{proof}

\section{Further Questions}
Not every one-relator group is a one-relator monoid. In fact, it was proved by 
Perrin and Schupp \cite{Perrin1984} that a one-relator group is positive if and only if it is a one-relation monoid. By a positive one-relator group we mean one which admits a one-relator group presentation where no inverse letter appears in the defining relator word.   
Baumslag \cite{baumslag1971positive}
proved that all positive one-relator groups are residually solvable. 
In that paper Baumslag points out that the Baumslag-Gersten group $G_2$ is not residually solvable. Therefore it is not a  
positive one-relator group and thus $G_2$ does not admit a one-relator monoid presentation.  
So the examples in this paper are not a source of examples of one-relator monoids with undecidable Diophantine problem, which leads us to ask: 

\begin{qu}
	Is the Diophantine problem decidable in one-relator monoids?
\end{qu}

The Baumslag-Gersten group is a famously badly behaved example of a one-relator
group. Perhaps if one restricts to a class of one-relator groups that excludes
the Baumslag-Gersten group due to one of its strange properties, one would
obtain a class of groups with decidable Diophantine problems. We consider,
for example, residual finiteness.

\begin{qu}
	Is the Diophantine problem decidable in residually finite one-relator
	groups?
\end{qu}

Since the conjugacy problem is an instance of the Diophantine problem with a single variable, this leads naturally to the following question. 

\begin{qu}
	It is decidable if a single equation with one variable in
	a given one-relator group is satisfiable? 
\end{qu}

\section*{Acknowledgements}
The authors would like to thank Corentin Bodart for pointing out \cite{BS-fixed}, thus
making the proof of \cref{BS-cents} substantially easier.

\bibliographystyle{custom} 
\bibliography{references}

\end{document}